\documentclass[11pt,leqno]{article}
\usepackage{graphicx, amsfonts, amsthm, amsxtra, amssymb, verbatim, makeidx}
\usepackage{subeqnarray, relsize}
\usepackage[mathscr]{euscript}
\usepackage[english]{babel}
\usepackage[fixlanguage]{babelbib}

\usepackage[utf8]{inputenc}
\usepackage[english]{babel}

\usepackage{wrapfig}
\usepackage{amssymb, amsmath, amsthm}
\usepackage{graphicx}
\usepackage{color}
\usepackage{amssymb}
\usepackage{url}
\usepackage{pdfpages}
\usepackage{fancyhdr}
\usepackage{subfig}
\usepackage{titlesec}
\usepackage{enumerate}
\usepackage{comment}
\usepackage{bigints}
\usepackage{diagbox}
\usepackage{cite}
\usepackage[fixlanguage]{babelbib}
\usepackage[unicode,
            psdextra,
            colorlinks=true,
            linkcolor=blue,
            citecolor=green
            ]{hyperref}
\usepackage[nameinlink,capitalize,noabbrev]{cleveref}

\makeindex
\makeglossary

\theoremstyle{definition}
\newtheorem{definition}{\bf Definition}[section]
\newtheorem{theorem}{Theorem}[section]
\newtheorem{lemma}{Lemma}[section]
\newtheorem{cor}{Corollary}[section]
\newtheorem{proposition}{Proposition}[section]

\newtheorem{remark}{Remark}[section]

\newtheorem{example}{Example}[section]

\renewenvironment{proof}{{\bfseries \noindent Proof} }{ \qed \\}

\begin{document}

\def\R{\mathbb{R}}                   
\def\Z{\mathbb{Z}}                   
\def\Q{\mathbb{Q}}                   
\def\C{\mathbb{C}}                   
\def\N{\mathbb{N}}                   
\def\H{{\mathbb H}}                
\def\A{\mathbb{A}}

\def\P{\mathbb{P}}
\def\Gal{\text{Gal}}
\def\GHod{\text{GHod}}
\def\SHod{\text{SHod}}
\def\Hod{\text{Hod}}
\def\res{\text{res}}
\def\prim{\text{prim}}
\def\dR{\text{dR}}
\def\Jac{\text{Jac }}
\def\Hom{\text{Hom}}

\def\ker{{\rm ker}}              
\def\GL{{\rm GL}}                
\def\ker{{\rm ker}}              
\def\coker{{\rm coker}}          
\def\im{{\rm Im}}               
\def\coim{{\rm Coim}}            
\def\NS{{\rm NS}}              
\def\End{{\rm End}}              
\def\rank{{\rm rank}}                
\def\gcd{{\rm gcd}}                  

\begin{center}
{\LARGE\bf On the Picard number and the extension degree of period matrices of complex tori
}
\footnote{ 
Math. classification: 14K20, 14K22, 32G20 11G10.
}
\\
\vspace{.25in} {\large{{\sc Robert Auffarth }}\footnote{Universidad de Chile, Facultad de Ciencias, Departamento de Matemáticas, Las palmeras 3425, Santiago, Chile,
{\tt rfauffar@uchile.cl}}}
 {{\sc and Jorge Duque Franco}}\footnote{ Dirección de Investigación, Vicerrectoría Académica, Instituto de Matemática y Fisíca, Universidad de Talca, Avenida Lircay s/n, Casilla 721, Talca, Chile, {\tt georgy11235@gmail.com}}
\end{center}


\begin{abstract}
The rank $\rho$ of the Néron-Severi group of a complex torus $X$ of dimension $g$ satisfies $0\leq\rho\leq g^2=h^{1,1}.$ The degree $\mathfrak{d}$ of the extension field generated over $\Q$ by the entries of a period matrix of $X$ imposes constraints on its Picard number $\rho$ and, consequently, on the structure of 
$X$. In this paper, we show that when $\mathfrak{d}$ is $2$, $3$, or $4$, the Picard number $\rho$ is necessarily large. Moreover, for an abelian variety $X$ of dimension $g$ with $\mathfrak{d}=3,$  we establish a structure-type result: $X$ must be isogenous to $E^g$, where $E$ is an elliptic curve without complex multiplication. In this case, the Picard number satisfies $\rho(X)=\frac{g(g+1)}{2}.$ As a byproduct, we obtain that if $\mathfrak{d}$ is odd, then $\rho(X)\leq\frac{g(g+1)}{2}.$
\end{abstract}

\section{Introduction}
\label{intro}
The Picard number $\rho$ is the rank of the Néron-Severi group $\NS(X)$ of a complex manifold $X,$ which parameterizes holomorphic line bundles on $X$ modulo analytic equivalence. When $X$ is a Kähler manifold, the Picard number satisfies 
$$\rho\leq h^{1,1}:=\dim H^1(X,\Omega_X^1).$$  
This number is a fundamental invariant of a variety and, in some cases, determines specific structural properties of the variety. For instance, in the case of K$3$ surfaces, if $\rho\geq 5,$ the surface admits an elliptic fibration and when $\rho\geq 12,$ such a fibration exists with a section (see \cite[\S 11.1]{huybrechts2016lectures}). Furthermore, if $\rho\geq 19,$ there is an isomorphism of integral Hodge structures between the transcendental lattices of $X$ and some abelian surface \cite{Morrison1984}. A similar phenomenon occurs for Hyperk\"{a}hler manifolds of K$3^{[n]}$-type with large Picard number, see \cite{prieto2024hyperk} for details.

Another more recent example arises in the context of Fano varieties. In \cite{casagrande2024}, it is shown that if $X$ is a smooth Fano $4$-fold with Picard number $\rho>12,$ then $X$ must be isomorphic to a product $S_1\times S_2,$ where $S_1$ and $S_2$ are del Pezzo surfaces. This result generalizes to dimension $4$ the analogous result for Fano $3$-folds established in \cite{mori1981classification}, which states that if $X$ is a smooth Fano $3$-fold with $\rho>5,$ then $X\simeq S\times \P^1$ where $S$ is a del Pezzo surface.

Another notable example arises in the case of abelian varieties, where structure theorems up to isogeny exist for those with large Picard numbers. Specifically, for $g\geq 5,$ we have 
$$
\rho(X)=(g-1)^2+1\Longleftrightarrow X\sim E_1^{g-1}\times E_2,$$ 
where $E_1$ has complex multiplication, and $E_1$ and $E_2$ are not isogenous. Similarly, for $g\geq 7,$ we have 
$$
\rho(X)=(g-2)^2+4 \Longleftrightarrow X\sim E_1^{g-2}\times E_2^2,$$ 
where $E_1$ and $E_2$ both have complex multiplication but are not isogenous. For more details, see \cite{hulek2019picard}. In the same paper, a structure result is presented for abelian varieties that achieve the largest possible Picard number among $g$-dimensional abelian varieties that split into a product of $r$ non-isogenous factors. More precisely, by Poincaré's Complete Reducibility Theorem, any abelian variety $X$ admits a decomposition of the form 
$$
X\sim A_1^{k_1}\times\dots\times A_r^{k_r}.
$$
Define $r(X):=r$ to be the number of non-isogenous simple factors that appear in the Poincaré decomposition. Now, let
$$
M_{r,g}:=\max \{\rho(X)|\dim X=g,\;\;r(X)=r\}.
$$
In this setting, we have
$$
\rho(X)=M_{r(X),g}\Longleftrightarrow X\sim E^{g-r+1}\times E_1\times\dots\times E_{r-1}, 
$$
where $E$ is an elliptic curve with complex multiplication that is not isogenous to any of the $E_i$'s, and the curves $E_i$ and $E_j$ are pairwise non-isogenous for $i\neq j.$

If a complex torus $X$ of dimension $g$ satisfies $\rho=h^{1,1}=g^2,$ then $X$ must be an abelian variety. Moreover, this occurs if and only if $X$ is isomorphic to a product of mutually isogenous elliptic curves with complex multiplication, which, in turn, happens precisely when the entries of a period matrix of $X$ lie in a quadratic extension of the rational numbers $\Q$. We revisit this last (well-known) fact in \cref{thm:NS1002}, and for a complete characterization, see \cref{theorem:max}. When this condition holds, we say that $X$ is (Picard) $\rho$-maximal. Complex tori/abelian varieties with this property exhibit interesting arithmetic and geometric properties; see, for instance, \cite{beauville2014some, shioda1977singular, shioda1974singular}. 

Our main theorem, closely related to the two previous examples, is a structural type result (up to isogeny) that further reinforces the phenomenon observed throughout the preceding examples.

\begin{theorem}
\label{thm:main}
Let $\Pi=(\tau\hspace{0.2cm}I_g)$ be a period matrix of an abelian variety $X$ of dimension $g\geq 2,$ where $\tau=(\tau_{ij})_{g\times g}\in M_g(\C)$ satisfies $\det(\mathrm{Im}(\tau))\neq0.$ Suppose the degree of the field extension generated by the entries of $\tau$ over $\Q$ is 
$$\mathfrak{d}=[\Q(\{\tau_{ij}\}):\Q]=3.$$ 
Then $X$ is isogenous to a self-product of an elliptic curve without complex multiplication. In this case, the Picard number is $\rho(X)=\frac{g(g+1)}{2}.$ 
\end{theorem}

\begin{remark}
    \cref{thm:main} provides a sufficient and rather elementary condition under which an abelian variety $X$ is forced to be isogenous to a self-product of an elliptic curve without complex multiplication. 
    By contrast, in \cite[Theorem~9]{wolfart2005triangle}, Wolfart establishes—under more technical assumptions—sufficient conditions ensuring that the Jacobian $\Jac(C)$ of a curve $C$ is isogenous to a self-product of a simple abelian variety. 
    In Wolfart’s setting, the simple factor may be either an elliptic curve or an abelian surface whose endomorphism algebra satisfies certain additional properties.
\end{remark}

The proof of \cref{thm:main} relies on the Poincaré Complete Reducibility theorem, the upper bound for the Picard number $\rho$ of a self-product of a simple abelian variety established by Hulek and Laface in \cite[Corollary 2.5]{hulek2019picard} and the following result, which provides a lower bound for $\rho.$ This theorem provides sufficient conditions for the Picard number $\rho$ of a complex torus $X$ to be large. More precisely:

\begin{theorem}
\label{thm:NS1002}
Let $\Pi=(\tau\hspace{0.2cm}I_g)$ be a period matrix of a complex torus $X$ of dimension $g,$ where $\tau=(\tau_{ij})_{g\times g}\in M_g(\C)$ satisfies $\det(\mathrm{Im}(\tau))\neq0.$ Let $\mathfrak{d}=[\Q(\{\tau_{ij}\}):\Q]$ denote the degree of the field extension generated by the entries of $\tau$ over $\Q.$ Then we have:
  \begin{enumerate}
      \item  $\mathfrak{d}=2\Longleftrightarrow X$ is $\rho$-maximal. In particular, $X$ is an abelian variety. 

      \item If $\mathfrak{d}=3,$ then $\frac{g(g+1)}{2}\leq \rho<g^2$.
      
      \item If $\mathfrak{d}=4,$ then $g\leq\rho<g^2$.  
  \end{enumerate}
\end{theorem}
The proof of \cref{thm:NS1002} relies on an estimate of the dimension of $\ker (T)$ over $\Q$, where $T$ is a certain matrix satisfying $\rho(X)=\dim_{\Q}\ker (T);$ see \cref{Th:NS} for further details. Now, if $X$ is an abelian variety and $\mathfrak{d}=2$ or $\mathfrak{d}=3$ we have a characterization of $X,$ as given in \cref{theorem:max} and \cref{thm:main} respectively. While a similar characterization is not available for $\mathfrak{d}=4$, by combining \cref{thm:NS1002} and \cref{prop:1303} we can further restrict the possible values of the Picard number $\rho$. By \cref{theorem:max} and the additivity of the Picard number, the more powers of elliptic curves with complex multiplication an abelian variety $X$ contains, the larger its Picard number. \cref{prop:1303} provides further evidence of this phenomenon. In particular, it implies that if the degree $\mathfrak{d}$ of the field extension over $\Q$ associated with the period matrix of $X$ is odd, then $\rho\leq\frac{g(g+1)}{2},$ see \cref{cor:1303}.

The structure of this paper is as follows: In \cref{sec2}, we introduce the key ingredients for our proofs. \cref{sec3} is dedicated to proving \cref{thm:main} and \cref{thm:NS1002}, as well as establishing several interesting consequences.

\bigskip

\noindent\textbf{Acknowledgements.} The authors are grateful to Roberto Villaflor for several stimulating conversations. The first author was supported by the Fondecyt ANID grant 1220997. The second author was supported by the Fondecyt ANID postdoctoral grant 3220631.

\section{Preliminaries}
\label{sec2}
In this section, we will show how to calculate the Picard number of a complex torus $X$ in terms of a period matrix for $X$. This approach allows us to establish lower bounds for the Picard number $\rho$ and, as an application, derive \cref{thm:NS1002}. We begin by computing the Picard number $\rho$ in terms of the rank over $\Q$ of a certain matrix $T$. We recall \cite[Proposition 1.1]{birkenhake1999complex} that $X$ has a period matrix of the form $(\tau\hspace{0.2cm}I_g)$, where $\det(\mathrm{Im}(\tau))\neq0$, and $T$ will be constructed from the entries of $\tau$. One key advantage of this approach is that, in concrete examples, $T$ can be implemented computationally using only the entries of $\tau.$ We will refer to both $\tau$ and $(\tau\hspace{0.2cm}I_g)$ as period matrices for $X$, depending on the context.

\begin{proposition}
\label{Th:NS}
    Let $\Pi=(\tau\hspace{0.2cm}I_g)$ be a period matrix of a complex torus $X$ of dimension $g,$ where $\tau=(\tau_{ij})\in M_g(\C)$ satisfies $\det(\mathrm{Im}(\tau))\neq0.$ Consider the linear map
    
    \begin{equation*}
 \begin{array}{cccc}
     T:\Q^{\frac{g(g-1)}{2}}\times\Q^{g^2}\times\Q^{\frac{g(g-1)}{2}}  & \longmapsto & \C^{\frac{g(g-1)}{2}} \\
      ((a_{ij})_{1\leq i<j\leq g},(b_{ij}),(c_{ij})_{1\leq i<j\leq g}) & \longmapsto & (w_{ij})_{i<j}
\end{array}
\end{equation*}
where   
$$
w_{ij}=a_{ij}+\sum_{k=1}^g(b_{jk}\tau_{ki}-b_{ik}\tau_{kj})+\sum_{1\leq l< k\leq g}c_{lk}(\tau_{kj}\tau_{li}-\tau_{lj}\tau_{ki}).
$$
with $i,j\in\{1,\dots,g\}$. Then the Picard number $\rho=\rank\NS(X)$ satisfies 
$$
\rho=2g^2-g-\rank_{\Q} (T).
$$
\end{proposition}

\begin{proof}
The Néron Severi group $\NS(X)$ can be identified with the matrix group 
\begin{equation*}
    \begin{pmatrix}
        A & B \\
        -B^{t} & C
    \end{pmatrix}
    \in M_{2g}(\R)
\end{equation*}
 where $A$ and $C$ are skew-symmetric matrices satisfying the condition 
 $$W:=A-B\tau+\tau^tB^t+\tau^tC\tau=0,$$ 
 see \cite[Proposition 1.3.4]{birkenhake1999complex}. The entries of the skew-symmetric matrix $W$ are explicitly given by
 $$
w_{ij}=a_{ij}+\sum_{k=1}^g(b_{jk}\tau_{ki}-b_{ik}\tau_{kj})+\sum_{1\leq l< k\leq g}c_{lk}(\tau_{kj}\tau_{li}-\tau_{lj}\tau_{ki})
 $$
where $a_{ij},$ $b_{ij},$ $c_{ij}$ are the entries of the matrices $A,$ $B,$ and $C$ respectively. Given that
$$
T((a_{i,j})_{i<j},(b_{ij}),(c_{ij})_{i<j})=(w_{ij})_{i<j},
$$
we conclude that $\rho=\dim_\Q \ker(T).$ Finally, applying the rank-nullity theorem, we obtain the desired result.   
\end{proof}

\begin{remark}
The matrix associated with the linear map $T$ can be expressed as

 \begin{equation*}
    T=\begin{pmatrix}
        Id_{\frac{g(g-1)}{2}\times \frac{g(g-1)}{2}} & \mathcal{B}_{\frac{g(g-1)}{2}\times g^2} & \mathcal{C}_{\frac{g(g-1)}{2}\times \frac{g(g-1)}{2}} 
    \end{pmatrix}
    \in M_{\frac{g(g-1)}{2}\times 2g^2-g}(\C)
\end{equation*}
 where the matrices $\mathcal{B}$ and $\mathcal{C}$ are defined as follows: \\

\noindent\textbf{Definition of $\mathcal{B}$:} Let $(i_0,j_0)\in \{1,\dots,g\}^2$ be indexed with the lexicographic order, and let $\tau=(\tau_{ij})_{g\times g}$ be the period matrix. The matrix $\mathcal{B}\in M_{\frac{g(g-1)}{2}\times g^2}(\C)$ is defined such that its columns $(i_0,j_0)$ are given by 
    
    \begin{equation*}
\mathcal{B}_{(i,j)(i_0,j_0)} = \left\lbrace
\begin{array}{ll}
\tau_{j_0i} & i<i_0 \\
-\tau_{j_0j}  & i_0<j\\
0 & \text{other cases.}
\end{array}
\right.
\end{equation*}
where $i<j\in \{1,\dots,g\}$.\\

\noindent\textbf{Definition of $\mathcal{C}$:} The matrix $\mathcal{C}\in M_{\frac{g(g-1)}{2}d}(\C)$ is defined by its entries $\mathcal{C}_{(i,j)(i_0,j_0)},$ for $i<j$ and $i_0<j_0$, given by of the determinant of the $2\times 2$ submatrix of $\tau$ formed by 
\begin{equation*}
    \begin{pmatrix}
        \tau_{i_0i} & \tau_{i_0j} \\
         \tau_{j_0i} & \tau_{j_0j}
    \end{pmatrix}.
\end{equation*}   
\end{remark}

\begin{example}
Using the description of the linear map $T$ from the previous remark, we can directly compute $T$ in the $2$-dimensional case. Specifically, we obtain
$$T=(1,-\tau_{12},-\tau_{22},\tau_{11},\tau_{21},\det (\tau)).$$
Consequently, we recover the well-known formula

$$\rho= 6-\dim_{\Q} \langle 1,\tau_{11},\tau_{12},\tau_{21},\tau_{22},\det \tau\rangle,$$ 
see \cite[$\S$2.7]{birkenhake1999complex}.  
\end{example}

By loosening the precision of \cref{Th:NS}, we derive lower bounds for the Picard number $\rho$ of a complex torus. These bounds provide simple criteria for identifying tori with large Picard numbers. We first define one of the fundamental notions of this article:

\begin{definition}
Let $\Pi=(\tau\hspace{0.2cm}I_g)$ be a period matrix of a complex torus $X$ of dimension $g,$ where $\tau=(\tau_{ij})_{g\times g}\in M_g(\C)$ satisfies $\det(\mathrm{Im}(\tau))\neq0.$ Define
$$F_X:=\mathbb{Q}(\{\tau_{ij}\})$$
   $$\mathfrak{d}_X=\mathfrak{d}:=[F_X:\Q]\in\mathbb{N}\cup\{\infty\}.$$ 
\end{definition}

Let us take a look at a few properties of $F_X$:

\begin{proposition}
\label{prop:2302}
Let $f:X\to Y$ be a map between complex tori with period matrices $(\tau\hspace{0.2cm}I_g)$ and $(\sigma\hspace{0.2cm}I_g)$, respectively. Then we have the following: 
\begin{enumerate}
\item If $f$ has finite kernel, then $F_X\subseteq F_Y$, and in particular $\mathfrak{d}_X\leq\mathfrak{d}_Y$.
\item If $f$ is surjective, then $F_Y\subseteq F_X$, and in particular $\mathfrak{d}_X\geq\mathfrak{d}_Y$.
\item If $f$ is an isogeny, then $F_X=F_Y$, and so $\mathfrak{d}_X=\mathfrak{d}_Y$.
\item $F_X=F_{X^\vee}$.
\end{enumerate}
In particular, 3. implies that $F_X$, and therefore $\mathfrak{d}_X$, only depends on $X$ and not on the specific period matrix chosen.
\end{proposition}

\begin{proof}
Clearly the third item follows from the previous two, and the fourth item is trivial since a period matrix for $X^\vee$ is simply $$(\tau^\tau\hspace{0.2cm}I_g).$$

To prove the first item, the homomorphism $f:X\to Y$ between complex tori induces the following equation:
\begin{equation}
\label{eq:0403}
    \rho_a(f)(\tau\hspace{0.2cm}I_g)=\left(\sigma\hspace{0.2cm}I_g\right)\rho_r(f),
    \end{equation}
where $\rho_a(f)$ is the analytic representation of $f$ and $\rho_r(f)$ is its rational representation. Writing 
$$\rho_r(f)=\left(\begin{array}{cc}A&B\\C&D\end{array}\right),\;\; A,B,C,D\in M_g(\mathbb{Z}),$$ 
we obtain in particular that
\[\rho_a(f)=\sigma B+D.\]
This implies that each coefficient of the matrix $\rho_a(f)$ belongs to the field $F_Y$. If $f$ has finite kernel, its analytic representation $\rho_a(f)$ is injective and therefore possesses a left inverse $M$ whose coefficients also lie in $F_Y$. Using \cref{eq:0403} we deduce that
\[\tau=M(\sigma A+C),\]
showing that each coefficient of $\tau$ also lies in $F_Y$. Therefore $F_X\subseteq F_Y$.

Now assume that $f$ is surjective. Then by dualizing, we get an injective morphism $f^\vee:Y^\vee\to X^\vee$, and so by the first and fourth items,
\[F_Y=F_{Y^\vee}\subseteq F_{X^\vee}=F_X.\]
\end{proof}

\begin{remark}
    \label{rem:2703}
    Let $Y$ be a complex subtorus of $X$, with period matrices $(\sigma\hspace{0.2cm}I_g)$ and $(\tau\hspace{0.2cm}I_g)$ respectively. A consequence of \cref{prop:2302} is that $F_Y$ is a subfield of $F_X$, and this only depends on the isogeny classes of $Y$ and $X.$ We will use this fact implicitly throughout the article in various arguments. 
\end{remark}

We now use the number $\mathfrak{d}$ to obtain certain bounds on the Picard number of a complex torus.

\begin{proposition}\label{NS1}
   Let $\Pi=(\tau\hspace{0.2cm}I_g)$ be a period matrix of a complex torus $X$ of dimension $g,$ where $\tau=(\tau_{ij})_{g\times g}\in M_g(\C)$ satisfies $\det(\mathrm{Im}(\tau))\neq0.$ For each pair $(i,j)$, let 
   $$\mathfrak{d}_{ij}=\dim_{\Q} \langle 1,\tau_{ki},\tau_{kj},\tau_{li}\tau_{kj}-\tau_{ki}\tau_{lj}\rangle_{l,k=1,\dots,g}$$ 
   be the dimension of the $\Q$-vector space generated by these elements. Then the Picard number $\rho$ of $X$ satisfies the inequalities
   \begin{equation}
    \label{eq:ineq2001}
   \rho\geq 2g^2-g-\sum_{1\leq i<j\leq g}\mathfrak{d}_{ij}\geq g^2-g(g-1)\left(\frac{\mathfrak{d}}{2}-1\right)  . 
   \end{equation}
\end{proposition}

\begin{proof}
According to \cref{Th:NS} the Picard number satisfies $\rho=\dim_\Q \ker(T).$ A vector $((a_{i,j})_{i<j},(b_{ij}),(c_{ij})_{i<j})\in\Q^{2g^2-g}$ belongs to $\ker(T)$ if and only if  
$$
T((a_{i,j})_{i<j},(b_{ij}),(c_{ij})_{i<j})=(w_{ij})_{i<j}=0,
$$
where
$$
w_{ij}=a_{ij}+\sum_{k=1}^g(b_{jk}\tau_{ki}-b_{ik}\tau_{kj})+\sum_{1\leq l< k\leq g}c_{lk}(\tau_{kj}\tau_{li}-\tau_{lj}\tau_{ki}).
 $$
This system consists of $\frac{g(g-1)}{2}$ equations in $2g^2-g$ variables, given by $w_{ij}=0$. The coefficients of each equation $w_{ij}=0$  belong to the $\Q$-vector space  
$$\langle 1,\tau_{ki},\tau_{kj},\tau_{li}\tau_{kj}-\tau_{ki}\tau_{lj}\rangle_{l,k=1,\dots,g}.$$ 
Thus, over $\Q,$ the number of independent equations is at most $\sum_{1\leq i<j\leq g}\mathfrak{d}_{ij},$ with the bound $\mathfrak{d}_{ij}\leq \mathfrak{d}.$ Therefore, the Picard number $\rho$ is at least the difference between the number of variables and the number of independent equations, yielding \cref{eq:ineq2001}.
\end{proof}

We conclude this section with an elementary inequality for positive numbers, which will be used in the proofs of \cref{thm:main} and \cref{prop:1303}.

\begin{lemma}
\label{lemma:1202}
Consider the quantity
$$g=\sum_{j=1}^{l}n_jk_j+\sum_{j=l+1}^{r}k_j,$$
where $n_j\geq 2$ and $k_j\geq1$. Then, the following inequality holds:
$$
2\sum_{j=1}^{l}n_jk_j^2+\sum_{j=l+1}^{r}k_j^2\leq g^2.
$$
Moreover, equality holds if and only if 
$$l=1\wedge n_1=2\wedge r=0 \text{ or } l=0\wedge r=1,$$
that is, if and only if $g=2k_1$ or $g=k_1.$
\end{lemma}
\begin{proof}
    Consider nonnegative integers $k_1,k_2$. Note that $(k_1+k_2)^2\geq k_1^2+k_2^2,$ with equality if and only if $k_1=0$ or $k_2=0.$ Using induction, it follows that for $k_j\geq 1,$
\begin{equation}
\label{ineq:2002}
\sum_{j=l+1}^{r}k_j^2\leq\left(\sum_{j=l+1}^{r}k_j\right)^2,
\end{equation}
with equality if and only if $l=0$ and $r=1.$ Taking this into account, for $n_j\geq 2$ and $k_j\geq1,$ we obtain
\begin{equation}
\label{ineq:20021}
2\sum_{j=1}^{l}n_jk_j^2\leq\sum_{j=1}^{l}n_j^2k_j^2\leq\left(\sum_{j=1}^{l}n_jk_j \right)^2
\end{equation}
with equality if and only if $l=1$ and $n_1=2.$ Combining inequalities in \cref{ineq:2002,ineq:20021} completes the proof.
\end{proof}

\section{Proof of Theorems}
\label{sec3}
In this section, we build upon the results from \cref{sec2} to prove the two main theorems of this article, namely \cref{thm:NS1002,thm:main}.

\begin{remark}
\label{rem:Kodaira}
Recall that a complex torus $X$ is called $\rho$-maximal if it satisfies $\rho=g^2.$ It turns out that every Picard maximal complex torus is an abelian variety. This is a special case of a more general phenomenon: any compact K\"{a}hler manifold $X$ that is $\rho$-maximal is necessarily projective. 

Indeed, the Kodaira embedding theorem states that $X$ is projective if and only if there exists some positive class $\eta\in\NS_{\Q}(X)=H^{1,1}(X)\cap H^2(X,\Q).$ In this context, the Picard number $\rho$ is bounded above by the Hodge number $h^{1,1},$ and being $\rho$-maximal means that $\rho=h^{1,1}.$ Note that the K\"{a}hler form $\omega$ is already a positive real $(1,1)$-class, ie., $\omega\in H^{1,1}(X)\cap X^2(X,\R)$ and is positive. If $X$ is $\rho$-maximal, then the space $\NS_{\Q}(X)$ is dense in $H^{1,1}(X)\cap X^2(X,\R).$ This guaranteees the existence of a rational $(1,1)$-class sufficiently close to the K\"{a}hler form $\omega$ that remains positive. By the Kodaira embedding theorem, this implies that $X$ is projective.      
\end{remark}

The first part of \cref{thm:NS1002}  states that $X$ is $\rho$-maximal if and only if the field extension associated with its period matrix $\tau$ has degree $\mathfrak{d}=2.$ This result was already known in the literature; see \cite[Exercise 5.6.10]{birkenhake2004complex}. Here, we provide an elementary proof of one direction of this equivalence and, for completeness,  also establish the other direction. \\

\noindent\textbf{Proof of \cref{thm:NS1002}}
Let us first prove Item 1. Observe that the inequality given by the extremes in \cref{eq:ineq2001} is equivalent to
    \begin{equation}
    \label{eq:ineq0306}
   \mathfrak{d}\geq 2\left(1+\frac{g^2-\rho}{g(g-1)} \right)\geq 2.
   \end{equation}
This implies that if $\mathfrak{d}=2$, then $g^2=\rho$, which occurs if and only if $X$ is $\rho$-maximal. Reciprocally, if $X$ is $\rho$-maximal, then by \cref{rem:Kodaira}  it is an abelian variety, and we may apply \cite[Proposition 3]{beauville2014some}. This yields
$$X\simeq E_1\times\dots\times E_g,$$ 
where the $E_i$ are elliptic curves with complex multiplication, all isogenous to each other. If $\langle \sigma_i,1\rangle$ is a lattice for $E_i$, then by \cref{prop:2302} we have $\Q(\sigma_i)=\Q(\sigma_j)$ for all $i,j$. Since the period matrix $(\sigma\hspace{0.2cm}I_g)$ of $E_1\times\dots\times E_g$ is given by $\sigma=\mathrm{diag}(\sigma_1,\ldots,\sigma_g)$, applying \cref{prop:2302} again, we obtain
$$
F_X=\Q(\sigma_1,\ldots,\sigma_g)=\Q(\sigma_1).
$$
Since $E_1$ has complex multiplication, the field $\Q(\sigma_1)$ has degree 2 over $\mathbb{Q}$, and thus $\mathfrak{d}=2$. 
For Items 2. and 3., the lower bounds follow directly from \cref{eq:ineq2001}, while the upper bound is a consequence of item 1.
$\blacksquare$

\begin{remark}
Recall that the algebraic dimension of a complex torus $X$ is defined as the transcendence degree of its field of meromorphic functions $\C(X)$: 
   $$ 
   a(X):=\text{tr }\deg_{\C}\C(X).
   $$
   It turns out that if the Picard number satisfies $\rho(X)=0$, then the algebraic dimension is $a(X)=0$. A simple yet insightful consequence of \cref{thm:NS1002} or \cref{eq:ineq0306} 
   is that a necessary condition for a complex torus $X$ to have Picard number zero is that $\mathfrak{d}\geq 5.$ More generally, if the Picard number satisfies $\rho<g$, then we still have $\mathfrak{d}\geq 5.$ 
\end{remark}

The first part of \cref{thm:NS1002} follows the same general principle as \cite[Proposition 2.6]{duque2023fake}, where the authors characterize the Fermat varieties $X_d^n$ whose group of algebraic cycles $H^n(X_d^n,\Z)_{\text{alg}}$ attains maximal rank $h^{n/2,n/2}$. This is a natural generalization of the concept of $\rho$-maximality. By combining the first part of \cref{thm:NS1002} with \cite[Proposition 3]{beauville2014some}  we obtain the following characterization of $\rho$-maximality in complex tori:

\begin{theorem}
\label{theorem:max}
Let be $X$ a complex torus of dimension $g.$ We have that
$$
\rank_{\Z}\; \End (X)\leq 2g^2
$$
and the following conditions are equivalent
\begin{itemize}
    \item[(i)] $X$ is $\rho$-maximal
    \item[(ii)] $\rank_{\Z}\; \End (X)= 2g^2$
    \item[(iii)] $X$ is isogenous to $E^g,$ where $E$ is an elliptic curve with complex multiplication. 
    \item[(iv)] $X$ is isomorphic to a product of mutually isogenous elliptic curves with complex multiplication. 
    \item[(v)] $\mathfrak{d}_X=2.$
\end{itemize}
\end{theorem}

The first statement follows from \cite[Theorem 1-1]{shimizu}.

\begin{remark}
\cref{theorem:max} and \cref{thm:main} ilustrates how the Picard number $\rho$ can dictate the structure of a complex torus/abelian variety (see  \cite{hulek2019picard} for further examples where the Picard number $\rho$ determines the structure of an abelian variety).
\end{remark}

In \cite[Corollary 2.5]{hulek2019picard}, Hulek and Laface established an upper bound for the Picard number $\rho$ of the self-product of a simple abelian variety. Specifically, let $A$ be a simple abelian variety of dimension $g$ and let $k\geq 1.$ Then
\begin{equation}
\label{eq:1102}
\rho(A^k)\leq \frac{1}{2}gk(2k+1).
\end{equation}
Using this inequality, we now derive an upper bound for the Picard number $\rho$ of an abelian variety $X$, which is independent of the field of definition of its period matrix.

\begin{proposition} 
\label{prop:1303}
Let $X$ be an abelian variety of dimension $g\geq 2,$ whose Poincaré decomposition is given by
    \begin{equation*}
        X\sim A_1^{k_1}\times\dots\times A_l^{k_l}\times E_{l+1}^{k_{l+1}}\times\dots\times E_{r}^{k_r},
    \end{equation*}
where $A_j$ are simple abelian and $E_j$ are elliptic curves. These factors are pairwise non-isogenous, and the $E_j$ are precisely the elliptic curves with complex multiplication appearing in this decomposition. In this setting, we have the bound
$$
\rho(X)\leq \frac{1}{2}\left(g(g+1)+\sum_{j=l+1}^{r}k_j(k_j-1) \right).
$$
In particular, if $X$ does not contain elliptic curves with complex multiplication, its Picard number satisfies $\rho(X)\leq\frac{g(g+1)}{2}$.
\end{proposition}

\begin{proof}
 Let $n_j$ denote the dimension of $A_j,$ and assume that for some $1\leq l_1\leq l,$ the abelian varieties $A_j$ are elliptic curves without complex multiplication for $l_1\leq j\leq l.$ Then, the Picard number of the powers of $E_j$ satisfies
 $$
 \rho(E_j^{k_j})=\frac{k_j(k_j+1)}{2}\; \text{ for } l_1\leq j \leq  l,$$
 and
 $$
 \rho(E_j^{k_j})=k_j^2 \; \text{ for } l+1\leq j \leq  r.
 $$
On the other hand, recall that the Picard number is additive (but not strongly additive), meaning that if $A_1$ and $A_2$ are non-isogenous simple abelian varieties, then
    $$
    \rho(A_1\times A_2)=\rho(A_1)+\rho(A_2),
    $$
    see \cite[Corollary 2.3]{hulek2019picard}. Thus, using the additivity of the Picard number together with \cref{eq:1102}, we obtain

\begin{equation*}
 \begin{split}
     \rho(X)&\leq\frac{1}{2}\sum_{j=1}^{l_1-1}n_jk_j(2k_j+1)+\frac{1}{2}\sum_{j=l_1}^{l}k_j(k_j+1)+\sum_{j=l+1}^{r}k_j^2\\
     &=\frac{1}{2}\left(g+2\sum_{j=1}^{l_1-1}n_jk_j^2+\sum_{j=l_1}^{l}k_j^2+2\sum_{j=l+1}^{r}k_j^2-\sum_{j=l+1}^{r}k_j  \right)\\
     &\leq \frac{1}{2}\left (g^2+g+ \sum_{j=l+1}^{r}k_j(k_j-1) \right) \text{ by \cref{lemma:1202}},
 \end{split}
 \end{equation*}
 where we have used the identity $g=\sum_{j=1}^{l_1-1}n_jk_j+\sum_{j=l_1}^{r}k_j.$
\end{proof}

\cref{prop:1303} enables us to deduce the following result, which further illustrates how the field over $\Q$ defined by the period matrix of an abelian variety $X$ imposes restrictions on the Picard number $\rho$.

\begin{cor}
    \label{cor:1303}
    Let $\Pi=(\tau\hspace{0.2cm}I_g)$ be a period matrix of an abelian variety $X$ of dimension $g,$ where $\tau=(\tau_{ij})_{g\times g}\in M_g(\C)$ satisfies $\det(\mathrm{Im}(\tau))\neq0.$ If $\mathfrak{d}$ is odd, then $\rho\leq\frac{g(g+1)}{2}.$  
\end{cor}

\begin{proof}
    Since $\mathfrak{d}$ is odd, it follows from \cref{prop:2302} (see also \cref{rem:2703}) that X does not contain any elliptic curve whose extension field associated is quadratic, i.e., it does not contain any elliptic curve with complex multiplication. Thus, the result follows immediately from \cref{prop:1303}.
 \end{proof}

In the context of \cref{thm:main}, if $\mathfrak{d}=3$, then using the second part of \cref{thm:NS1002} and \cref{cor:1303}, we already have $\rho(X)=\frac{g(g+1)}{2}$. Thus, it remains to prove the structure result. To this end, we will use the same idea as \cref{prop:1303}, but here we will refine the inequalities further. The proof is presented below:\\

\noindent\textbf{Proof of \cref{thm:main}}
By Poincaré's Complete Reducibility Theorem \cite[Thm. 5.3.7]{birkenhake2004complex}, the abelian variety $X$ admits the decomposition
    \begin{equation}
    \label{eq;dec0602}
        X\sim A_1^{k_1}\times\dots\times A_l^{k_l}\times E_{l+1}^{k_{l+1}}\times\dots\times E_{r}^{k_r},
    \end{equation}
    where $A_j$ are simple abelian varieties of dimensions $n_j\geq 2,$ $E_j$ are elliptic curves that are pairwise non-isogenous, and $k_j\geq 1$. This decomposition is unique up to isogenies and permutations. Since $\mathfrak{d}=3$, the extension field over $\Q$ determined by the period matrix of each elliptic curve $E_j$ has also degree $3$, by \cref{prop:2302}. Consequently, $E_j$ does not admit complex multiplication. 

    On the other hand, recall that the Picard number is additive (but not strongly additive), meaning that if $A_1$ and $A_2$ are non-isogenous simple abelian varieties, then
    $$
    \rho(A_1\times A_2)=\rho(A_1)+\rho(A_2),
    $$
    see \cite[Corollary 2.3]{hulek2019picard} Using this fact, along with the second part of \cref{thm:NS1002}, \cref{eq:1102}, and the identities $$g=\sum_{j=1}^{l}n_jk_j+\sum_{j=l+1}^{r}k_j, \;\;\;\rho(E_j^{k_j})=\frac{k_j(k_j+1)}{2},$$ 
    we obtain the inequality 
\begin{equation*}
    \frac{g(g+1)}{2}\leq\rho(X)=\sum_{j=1}^{l}\rho(A_j^{k_j})+\sum_{j=l+1}^{r}\rho(E_j^{k_j})\leq \frac{1}{2}\left(g+ 2\sum_{j=1}^{l}n_jk_j^2+\sum_{j=l+1}^{r}k_j^2 \right),
    \end{equation*}
from which it follows that
\begin{equation}
\label{eq:1202}
g^2\leq 2\sum_{j=1}^{l}n_jk_j^2+\sum_{j=l+1}^{r}k_j^2.
\end{equation}
By virtue of \cref{lemma:1202}, equality holds in \cref{eq:1202}. Applying \cref{lemma:1202} once more, we conclude that the decomposition in \cref{eq;dec0602} must take the form 
$$
X\sim A^{\frac{g}{2}} \text{ or } X\sim E^{g}, 
$$
where $A$ is a simple abelian surface and $E$ is an elliptic curve without complex multiplication. 

Now if $A$ is a simple abelian surface such that $\frak{d}_A=3$ (which would be the case if $\frak{d}=3$ and $X\sim A^{\frac{g}{2}}$), then let $\tau\in\mathbb{C}\backslash\mathbb{R}$ be an element that appears as an entry in a period matrix of $A$. Define $Y$ to be the elliptic curve given by the lattice $\langle 1,\tau\rangle$. Then $\frak{d}_{A\times Y}=3$, but $A\times Y$ does not obey the classification result above. Therefore, we conclude that such an example does not exist. In particular, if $\frak{d}=3$, then $X$ is isogenous to a self-product of an elliptic curve.
$\hfill \blacksquare$\\

\begin{remark}
From the previous proof, it follows that if an abelian variety $X$ has Picard number $\rho(X)=\frac{g(g+1)}{2}$ and does not contain elliptic curves with complex multiplication then $X$ satisfies the same structure result as \cref{thm:main}.
\end{remark}




\bigskip



\bibliographystyle{alpha}

\bibliography{ref}



\end{document}